\documentclass[11pt,twoside]{amsart}
\usepackage{pgfplots}
\pgfplotsset{compat=1.15}
\usepackage{mathrsfs}
\usetikzlibrary{arrows}
\usepackage{amsmath, amssymb, amsfonts, amsthm,latexsym,,mathtools,,hyperref,fancyhdr,comment, tabularht, tabularx, longtable, mathrsfs, cleveref,booktabs, color, geometry,bm}
\usepackage{enumitem}
\usepackage{multicol}
\usepackage{hyperref,url}
\usepackage{tikz}
\tikzstyle{vertex}=[circle, draw, inner sep=0pt, minimum size=6pt,fill=black]

\long\def\symbolfootnote[#1]#2{\begingroup%
	\def\thefootnote{\fnsymbol{footnote}}\footnote[#1]{#2}\endgroup}

\allowdisplaybreaks
\def \N {{\mathbb{N}}}
\def \Z {{\mathbb{Z}}}

\newtheorem*{theorem*}{Theorem}
\newtheorem{theorem}{Theorem}[section]

\newtheorem{lemma}[theorem]{Lemma}

\newtheorem*{ex*}{Example}

\date{}

\begin{document}
\title{Groups Having 13 Cyclic Subgroups}
	\author{Khyati Sharma}
	\address{Indian Institute of Science Education and Research, Berhampur, Odisha, India, 760003.}
	\email{khyatisharma0907@gmail.com}

	\today 

\begin{abstract}
   In [{\em Finite groups with a small number of cyclic subgroups, Czechoslovak Mathematical Journal 75.3 (2025): 839-851}], the authors classified all the groups having $13$ cyclic subgroups. In this paper, we give a new proof for the classification of all finite groups having exactly $13$ cyclic subgroups.
\end{abstract}

\maketitle
\section{Introduction}
 The problem of counting cyclic subgroups of a finite group has been studied actively over the last couple of decades. Miller was the first one to give a result in this direction. In \cite{miller1929number}, Miller proved that if $G$ is a finite group, then
    \begin{equation}\label{eq1}
|G|=\sum\limits_{m||G|}c(m)\varphi (m),
\end{equation}
where $c(m)$ denotes the number of cyclic subgroups of order $m$ in $G$, and $\varphi$ is the Euler totient function. If $G$ is a group, then $c(G)$ denotes the number of cyclic subgroups of $G$. A finite group $G$ is $n$-cyclic if $c(G)=n$.

Richard's Theorem~\cite{richards1984remark} provides a lower bound on $c(G)$, in particular $c(G)\geq d(|G|)$, where $d(|G|)$ denotes the number of positive divisors of $|G|$ and the equality holds if and only if $G$ is cyclic. Zohu~\cite{zhou2016finite}, Kalra~\cite{kalra2019finite}, Ashrafi and Haghi~\cite{ashrafi2019n} classified all the $n$-cyclic groups for $3\leq n\leq 10$. In our earlier papers~\cite{11cyclic,sharma2025groups}, we classified all $11$ and $12$-cyclic groups. Recently, Liu et al.~\cite{13-Cyclic} classified all $13$-cyclic groups. In this article, we give an alternative proof for the classification of all $13$-cyclic groups. We have divided the proof into two parts, depending on when $G$ is a nilpotent or a non-nilpotent group. Then, with the help of many equivalent conditions given in~\cite[Chapter 10]{rose2009course}, for a group to be nilpotent, we give the classification of all $13$-cyclic groups.
\section{Notations}
The notations $\Z_n$, $D_{2n}$, $Q_{2^n}$, $S_n$, $A_n$, $SL(2,3)$ and $Dic_{n}$ denote the cyclic group of order $n$, dihedral group of order $2n$, generalized quaternion group of order $2^n$, symmetric group with $n$ symbols, alternating group with $n$ symbols, special linear group of order $24$ and dicyclic group of order $4n$, respectively. Throughout this paper, $p, q$ and $r$ are distinct prime numbers, and the Euler totient function is denoted by $\varphi$. Here, $\N$ denotes the set of natural numbers and does not contain zero. The number of Sylow $p$-subgroups of $G$ is denoted by $n_p(G)$. A group $G$ is {\em Dedekind} if every subgroup of $G$ is normal. A non-cyclic group $G$ is said to be {\em minimal non-cyclic} if all its proper subgroups are cyclic.\\
Let $d(n), \omega(n)$ denote the number of positive divisors and the number of distinct prime divisors of $n$, respectively. In this paper, all the calculations for small groups are done by using GAP~\cite{GAP4} and the {\em SmallGroup$(n, i)$} denotes the $i^{th}$ group of order $n$ in the {\em Small Group Library} of GAP.
\section{Classification of 13-cyclic groups}\label{sec:13cyclicgroups}
  Suppose $G$ is a $13$-cyclic group of order $n$. An immediate consequence of Richard's Theorem~\cite{richards1984remark} is that $\omega(n)\le 3$ and $n$ is one of the form $p^k, pq, p^2q, p^3q, p^4q, p^5q, p^2q^2, p^3q^2, pqr$ or $p^2qr$, where $k\leq {12}$. Now we classify all $13$-cyclic groups. We divide the proof into two parts, depending on when $G$ is nilpotent or non-nilpotent. The following lemma is an immediate consequence of nilpotent groups.  
 \begin{lemma}\label{Nil}
     Let $G$ be a nilpotent group with $c(G)$ be a prime number. Then $G$ is a $p$-group.
 \end{lemma}
 \begin{proof}
     The proof is immediate by using the fact that a nilpotent group is the direct product of its Sylow subgroups and by using~\cite[Proposition 2.2]{CyclicityDegree}.
 \end{proof}
 \begin{lemma}
     If $G$ is a nilpotent group with $c(G)=13$, then $G\cong \Z_{p^{12}}, \Z_{11}\times \Z_{11}$ and $Q_{32}$.
 \end{lemma}
 \begin{proof}
    Since $c(G)=13$, then by Lemma~\ref{Nil}, $G$ is a $p$-group. Let $|G|=p^a$ and $M$ be a maximal subgroup of $G$. First, suppose that $M$ is cyclic. Then by following the proof of~\cite[Lemma 3.2] {sharma2025groups}, $G\cong \Z_{p^{12}}, \Z_{11}\times \Z_{11}$ and $Q_{32}$. Further, assume that $M$ is non-cyclic, then also by following the proof of~\cite[Lemma 3.2] {sharma2025groups}, we get $G\cong Q_{32}$. Hence, the result holds.
 \end{proof}
\begin{lemma}
    Let $G$ be a non-nilpotent group such that $c(G)=13$ and $|G|=p^aq$, where $a=1$ or $2$. Then $G\cong D_{22}, \Z_{11}\rtimes \Z_{5}={ SmallGroup(55, 1)}$.
\end{lemma}
\begin{proof}
 If $|G|=pq$ and $p<q$, then by \cite[Lemma 3.1]{kalra2019finite}, $c(G)=q+2$. Therefore $q=11$ and $G\cong D_{22},~ \Z_{11}\rtimes \Z_{5}$.\\
If $|G|=p^2q$, there are the following two sub-cases:
 \begin{enumerate}
        \item If $p<q$, then according to Proposition~$3.2$ of \cite{kalra2019finite}, $c(G)\in \{6, 2p+4, pq+4, q+4, 2q+2\}$. It is easy to see that there are no prime numbers $p$ and $q$ such that $c(G)=13$. 
        \item If $p>q$, then by Proposition~$3.2$ of \cite{kalra2019finite}, $c(G)\in \{6, 2p+4, p^2+3, p^2+p+2, 2p+3, 3p+2\}.$ Thus $c(G)=13$ is possible only if $p=5$ and $q=2$ or $3$ and $|G|=50, 75$. By using GAP~\cite{GAP4}, no such group of these orders is $13$-cyclic.
    \end{enumerate}
\end{proof}
 \begin{lemma}
     Let $G$ be a non-nilpotent group such that $c(G)=13$ and $|G|=p^{a}q$, where $a=3,4$ or $5$. Then $G\cong SL(2,3), \Z_7\rtimes \Z_8, \Z_7\rtimes \Z_{27}, \Z_5\rtimes \Z_{16}$ and $\Z_3\rtimes \Z_{32}$.
 \end{lemma}
 \begin{proof}
    First, assume that a Sylow $p$-subgroup of $G$ is cyclic. If the Sylow $q$-subgroup is normal in $G$, then Sylow $p$-subgroup is not normal in $G$ as $G$ is non-nilpotent. Thus, by Sylow theorem $n_p(G)\geq {1+p}$. Moreover, either $G$ has a unique subgroup of order $p$, or they are at least $1+p$. If the subgroup of order $p^k$ is unique, then $G$ has a unique cyclic subgroup of order $p^kq$. Also, $n_p(G)=q$. Therefore, by using the fact $c(G)=13$, we get $|G|\in \{2^3.3,2^3.5,2^3.7,3^3.2,3^3.5,3^3.7,5^3.2,
5^3.3,5^3.7,2^4.3,2^4.5,3^4.2,3^4.5,5^4.2,5^4.3,2^5.3\}$. If the subgroup of order $p$ is not unique, then we also get the same possibilities for $|G|$. We need to look only among non-nilpotent, supersolvable groups $G$ of the above orders. It can be checked using GAP, that $G\cong \Z_{7}\rtimes \Z_8={ SmallGroup(56, 1)}, \Z_{7}\rtimes \Z_{27}={ SmallGroup(189, 1)}, \Z_{5}\rtimes \Z_{16}={ SmallGroup(80, 1)}$ or $\Z_3\rtimes \Z_{32}={ SmallGroup(96, 1)}$. Now assume that Sylow $q$-subgroup is not normal in $G$. Then, by Sylow theorem $n_q(G)\geq {1+q}$ and $n_q(G)\geq p$. This implies that $p,q\leq 5$. Now, it can be checked by using GAP; no such group of these orders is $13$-cyclic.

         Now assume that Sylow $p$-subgroup of $G$ is not cyclic. So, there are the following two cases. First, suppose that the Sylow $q$-subgroup of $G$ is not normal. By \cite[Theorem 1.1]{LowerBoundNumberofCyclicSubgroups}, if $p^3\not =8$, then $c(P)\geq (a-1)p+2$ and $(a-1)p+2+q+1\leq c(G)=13$, where $(q+1)$, denotes cyclic subgroup of order $q$. This implies that $(a-1)p+q\leq 10$. If $p^3=8$, then by \cite[Theorem 1.1]{LowerBoundNumberofCyclicSubgroups}, $5+q+1\leq 13$ that is $q\leq 8$. Therefore $|G|\in \{2^3.3, 2^3.5, 2^3.7, 2^4.3, 2^4.5, 2^5.3, 3^3.2, 3^3.5, 3^4.2\}$. By using GAP, one can verify that $G\cong SL(2,3)$. Further assume that Sylow $q$-subgroup is normal in $G$. Then Sylow $p$-subgroup is not normal in $G$. Also, by \cite[Theorem 1.1]{LowerBoundNumberofCyclicSubgroups}, we get $(a-1)p+2+1=c(G)\leq 13$, where one denotes a unique subgroup of order $q$. Thus $p\leq 5$. Since the Sylow $q$-subgroup of $G$ is normal, then $G$ contains a maximal subgroup of order $p^{a-1}q$ say $M$, such that $c(M)\leq 11$ by~\cite[Proposition 2.4]{11cyclic}. Furthermore, $M\in \{\Z_{p^2q},\Z_{p^3q},\Z_{p^4q},\Z_3\times \Z_4,\Z_2\times \Z_{2q},A_4,\Z_5\rtimes\Z_4,\Z_3\times \Z_{3q},D_{12},Dic_7, \Z_7\rtimes \Z_9,\Z_3\times S_3,\Z_5\rtimes \Z_8, \Z_3\rtimes \Z_{16}, \Z_3\rtimes \Z_8, \Z_p\times Q_8\}$ by~\cite[Theorem 1.2, 2.4]{ashrafi2019n} and~\cite[Theorem 1.1]{11cyclic}. Now, we discuss all these cases separately. Let $M\cong \Z_{p^2q}$. Then $|G|=p^3q$ and $c(M)=6$. If $G$ contains a unique subgroup of order $p$, then the Sylow $p$-subgroup is isomorphic to $Q_8$. Since the Sylow $p$-subgroup is not unique, we can check that either $|G|=24, 40$ or $c(G)>13$. If the subgroup of order $p$ is not unique, then they are at least $1+p$, then $1+p+5\leq 13$. Thus $p=2,3,5$ or $7$. Also, a Sylow $p$-subgroup contains a cyclic subgroup of order $p^2$. By using GAP for each such choice of Sylow $p$-subgroup for $p=2,3,5$ and $7$, we can check that $c(G)\not =13$. Moreover, in a similar way for other choices of $M$, we can either find the values of $p$ and $q$, or we can explicitly count the number of cyclic subgroups of $G$, and they are more than $13$. Hence, no such group is $13$-cyclic. 
          \end{proof}
         \begin{lemma}
             There does not exist any non-nilpotent group $G$ such that $c(G)=13$ and $|G|=p^aq^2$ for $a=2,3$.
         \end{lemma}
         \begin{proof}
             If $|G|=p^2q^2$, then Sylow $p$ and $q$ subgroups of $G$ are Dedekind. Now, we prove the result by examining the following cases.
  \begin{enumerate}
    \item {\em \textbf{$\mathbf{G}$ has a unique subgroup of orders $\mathbf{p}$ and $\mathbf{q}$.}} In this case Sylow $p$ and $q$ subgroups of $G$ are cyclic. If $n_p(G)=n_q(G)=1$, then $G$ is cyclic, which is a contradiction. Also, by Sylow theorems, $n_p(G), n_q(G)>1$ is not possible. Let us assume that $n_p(G)>1$ and $n_q(G)=1$. Then by Sylow theorem, $n_p(G)\geq {1+p}$ and $n_p(G)\in \{q,q^2\}$. Moreover, $G$ has a unique cyclic subgroup of orders $1, p, q, pq$ and $q^2$.\\
            If $n_p(G)=q$, then $p<q$ and $c(G)\geq {q+5}$. Hence $p\in \{2,3,5\}$ and $q\in \{2,3,5,7\}$. Also $|G|\in \{36,100,196,225,441,1225\}$. If $n_p(G)=q^2$, then either $|G|=36$ or $c(G)>13$. A simple calculation with GAP~\cite{GAP4} shows that no such group of these orders is $13$-cyclic. 
            \item {\em \textbf{$\mathbf{G}$ has a unique subgroup of order $\mathbf{q}$ and at least $\mathbf{p+1}$ subgroups of order $\mathbf{p}$}.} This implies that Sylow $q$-subgroup of $G$ is cyclic. If Sylow $p$ subgroup of $G$ is normal, then $G$ has at least $q+1$ cyclic subgroups of order $q^2$. This implies that $c(G)\geq 1+(p+1)+(q+1)+1$, (here $p+1$ is the number of subgroups of order $p$, $q+1$ is the number of subgroups of order $q^2$, first $1$ denotes the trivial subgroup, and last $1$ is for a subgroup of order $q$). Therefore,  $p+q\leq 9$ and $G\in \{36, 100, 196, 225\}$. By using GAP, no such group of these orders is $12$-cyclic. If Sylow $p$ subgroup is not normal, then by \cite[Lemma 2.2]{11cyclic} $q\leq 11$. Also, by Sylow's theorem, $G$ has at least $2p+2$ cyclic subgroups of order $1,p$, and $p^2$. Thus $2p+2+2\leq c(G)$, here the last $2$ denotes the number of cyclic subgroups of order $q$ and $q^2$. Thus $G\in \{36,100,196,225,441,484,1089\}$ as $c(G)=13$. By using GAP, no such group of these orders is $13$-cyclic. 
            \item {\em \textbf{Subgroups of orders $\mathbf{p}$ and $\mathbf{q}$ are not unique.}} By \cite{frobenius1895verallgemeinerung}, $G$ has at least $p+1$ and $q+1$ subgroups of order $p$ and $q$, respectively. Therefore, $c(G)\geq p+q+3$. Accordingly, $p+q\in \{2,3,4,5,6,7,8,9\}$ and $|G|\in\{36,100,196,225,441\}.$ By GAP~\cite{GAP4}, after analyzing 
 all the groups of order $36,100,196,225$ and $441$, it can be seen that none of them is $13$-cyclic.
 \end{enumerate}
 If $|G|=p^3q^2$, then Sylow $q$-subgroup is Dedekind. We have the following cases.
 \begin{enumerate}
     \item {\em \textbf{Sylow $\mathbf{q}$ subgroup is normal in $\mathbf{G}$.}} Then Sylow $p$-subgroup is not normal in $G$. Also, $G$ contains a maximal subgroup of order $p^2q^2$, say $H$, such that $c(H)\leq 11$ by~\cite[Proposition 2.4]{11cyclic}. Since $c(H)\leq 11$, then $H\cong Z_{p^2q^2}$ and $c(H)=9$ by~\cite[Theorem 1.2, 2.4]{ashrafi2019n} and~\cite[Theorem 1.1]{11cyclic}. First, suppose that the Sylow $p$-subgroup of $G$ is cyclic. Then by the Sylow theorem, $n_p(G)\geq 1+p$ and $n_p(G)\geq q$. This implies that $p,q\in \{2,3\}$ and $|G|=24, 54$. Using GAP, one can check that no such group of the above orders is $13$-cyclic. Next, we assume that Sylow $p$-subgroup of $G$ is not cyclic. Then by \cite[Theorem 1.1]{LowerBoundNumberofCyclicSubgroups} either $p=2$ or $G$ has at least $2p+2$ cyclic subgroups of order $1,p$ and $p^2$. Thus if $p\not =2$, then $2p+2+2\leq c(G)$, here last $2$ denotes the number of cyclic subgroups of order $q^k$, where $k\geq 1$. This implies that $p=2$ or $3$ as $c(G)=13$. If $p=3$, then any Sylow $p$-subgroup $P$ of $G$ is either isomorphic to $\Z_9\times \Z_3$ or $\Z_9\rtimes \Z_3$. Now by counting the cyclic subgroups of $G$, it is easy to see that $c(G)>13$.
     \item {\em \textbf{Sylow $\mathbf{q}$ subgroup is not normal in $\mathbf{G}$.}} Then there are the following cases.  assume that, both the Sylow $p$ and $q$-subgroups of $G$ are cyclic. Then by the Sylow theorem, $n_p(G)\geq 1+p$ and $n_q(G)\geq 1+q$. Therefore $p,q\in \{2,3,5\}$. If both $p$ and $q$ sylow subgroups are non-cyclic, then by~\cite[Theorem 1.1]{meng2024finite}, either $p=2$ or $G$ has at least $2p+2$ cyclic subgroups of order $p^k$ and $q+2$ cyclic subgroups of order $q^k$, where $k\geq 1$. Thus, $|G|\in \{72,108,200,675\}$. By using GAP, one can check that no such group of above orders is $13$-cyclic. Similarly, we can handle the case when one of the Sylow subgroups is cyclic, and the other is non-cyclic. 
 \end{enumerate}
 \end{proof}
 \begin{lemma}
     There does not exist any non-nilpotent group $G$ such that $c(G)=13$ and $|G|=pqr$.
 \end{lemma}
 \begin{proof}
     If $|G|=pqr$, where $p<q<r$, then $n_r(G)=1$, i.e., the Sylow $r$-subgroup, $R$ and the Hall subgroup $H_{q,r}$ of order $qr$ are normal in $G$, i.e., $G\cong H_{q,r}\rtimes \mathbb{Z}_p$ and $n_p(G) \in \{q,r,qr\}$, i.e., $n_p(G)\geq q$. If $H_{q,r}\cong \mathbb{Z}_r \rtimes \mathbb{Z}_q$, then $c(H_{q,r})=r+2$. Now consider a Hall $\{p,q\}$-subgroup, $H_{p,q}$ of $G$. If it is normal in $G$, then $G\cong H_{p,q}\times R$. This contradicts that $c(G)=13$ is a prime number. Thus, any Hall $\{p,q\}$-subgroup is not normal in $G$. Now, counting the cyclic subgroups of $G$, we get $$13=c(G)\geq n_p(G)+(r+2), \mbox{ i.e., }q+r\leq 11.$$
This leaves very few choices for $(p,q,r)$ and can be checked exhaustively in GAP; no such group has exactly $13$ cyclic subgroups. Thus $H_{q,r}\cong \mathbb{Z}_{qr}$ and $c(H_{q,r})=4$. If $n_p(G)=r$ or $qr$, we get $r+4\leq 13$ that is $r\leq 7$. Again, using GAP, one can check that no such group exists. So, we are left with $n_p(G)=q$. Now, we try to evaluate $n_q(G)$. Let $T$ be the unique subgroup of order $q$ in $H_{q,r}$. If $T'$ is any other subgroup of order $q$ in $G$, then $|T'\cap H_{q,r}|=1$ and $T'H_{q,r}=G$ (as $H_{q,r}$ is a normal, maximal subgroup of $G$). Thus $|G|=|T'H_{qr}|=q^2r>pqr$, a contradiction. Thus $G$ has a unique subgroup of order $q$, i.e., $n_q(G)=1$. Now, we consider a Hall $\{p,r\}$-subgroup $H_{p,r}$ of $G$. It is either cyclic or isomorphic to $\mathbb{Z}_r \rtimes \mathbb{Z}_p$. In the latter case, $H_{p,r}$ contains $r>q$ subgroups of order $p$, a contradiction. Thus $H_{p,r}$ is cyclic. Similarly, if $H_{p,q}$ is a Hall $\{p,q\}$-subgroup of $G$ and $H_{p,q} \cong \mathbb{Z}_q\rtimes \mathbb{Z}_p$, then $H_{p,q}$ contains all the $q$ many subgroups of order $p$ in $G$. Let $P$ be the subgroup generated by all elements of order $p$ in $G$. Clearly $P$ is a characteristic subgroup of $G$ and $\mathbb{Z}_p \subsetneq P \subseteq H_{p,q}$, i.e., $H_{p,q}=P\lhd G$. However, we have proved earlier that any Hall $\{p,q\}$-subgroup is not normal in $G$. Thus $H_{p,q}$ is also a cyclic group. Hence, all proper subgroups of $G$ are cyclic, i.e., $G$ is a finite minimal non-cyclic group. Such groups are classified in \cite{miller1906number, redei1947schiefe}, and their orders have at most two distinct prime divisors, which is a contradiction.
 \end{proof}
 \begin{lemma}
      There does not exist any non-nilpotent group $G$ such that $c(G)=13$ and $|G|=p^2qr$.
 \end{lemma}
 \begin{proof}
    We will discuss the proof case by case.
    \begin{enumerate}
        \item \textbf{All the Sylow subgroups of $\mathbf{G}$ are cyclic.} Then by \cite[Theorem 1.2]{bray1982between}, $G$ is metacyclic and hence supersolvable. Thus $G$ contains maximal subgroups of orders $p^2q, p^2r$ and $pqr$. By \cite[Proposition 2.4]{11cyclic}, any maximal subgroup of these orders has at most $11$ cyclic subgroups. Moreover, all these maximal subgroups are either cyclic or they are isomorphic to $\Z_3\rtimes \Z_4, \Z_5\rtimes \Z_4, \Z_p\times S_3, D_{12},Dic_7,\Z_7\rtimes \Z_9$ and $\Z_3\times S_3$ by~\cite[Theorem 1.2, 2.4]{ashrafi2019n} and~\cite[Theorem 1.1]{11cyclic}. Also, all these maximal Subgroups of $G$ can not be cyclic, otherwise $G$ will become minimal non-cyclic, which is not possible as the order of any non-cyclic minimal non-cyclic group has at most two prime divisors by~\cite{miller1906number, redei1947schiefe}. Let $H$ be a non-cyclic maximal subgroup of $G$. Now we discuss all the possibilities of $H$ separately.\\
        If $H\cong Dic_7, \Z_7\rtimes \Z_9$ and $\Z_3\times S_3$, then $c(H)=11$. Without loss of generality, assume that $|H|=p^2q$. Since $c(G)=13$, then the Sylow $r$-subgroup of $G$ is unique. This implies that $G$ can not have a unique subgroup of order $p^2q$ as $c(G)=13$. By simple calculation, we can see that no such group is $13$-cyclic.\\
        Further, assume that $H\cong \Z_p\times S_3$ or $D_{12}$. If $H\cong \Z_p\times S_3$, then $c(H)=10$. Now, the following possibilities exist.\\
        If $|G|=6p^2$, where $p\geq 5$ then $G$ has a unique cyclic subgroup of order $p^2$. Also, $G$ can not have a normal subgroup of order $6$, otherwise $G\cong \Z_{p^2}\times H_6$ and $c(G)\not =13$, where $H_6$ denotes the subgroup of order $6$ in $G$. This is a contradiction. If other subgroups of order $6$ are cyclic, then they are at least two. Thus $G$ has at least $c(H)+1+2=10+1+2=13$ cyclic subgroups, where $1$ denotes a cyclic subgroup of order $p^2$ and $2$ denotes $G$ has two cyclic subgroups of order $6$. By equation~\ref{eq1} one can easily check that this is not possible for a group of order $6p^2$. If $G$ has at least two subgroups isomorphic to $S_3$, then by taking different possibilities on their intersection, one can see that $G$ is not $13$-cyclic.\\
        If $|G|=12p$, then also $G$ has a unique subgroup of order $p$, as $p\geq 5$. Since the Sylow $2$-subgroup of $G$ is cyclic, it can not be unique. In this case, the number of cyclic subgroups of $G$ is $c(H)+3=10+3$, where $3$ denotes the number of cyclic subgroups of order $4$. By equation~\ref{eq1}, one can check that such a group can not have $13$ cyclic subgroups.\\ If $|G|=18p$, then with similar arguments, one can check that no group of order $18p$ is $13$-cyclic.\\ 
If $H\cong D_{12}$, then $c(H)=10$ and $|G|=12p$. Thus $G$ has a unique subgroup of order $p$, as $p\geq 5$ and $H$ is also a unique subgroup of $G$ of order $12$. This implies that $G$ has at least $c(H)+3=10+3$ cyclic subgroups, where $3$ denotes the cyclic subgroup of order $p, 3p$, and $6p$. By equation~\ref{eq1}, we can check that this group is not $13$-cyclic.\\
 If $H\cong \Z_5\rtimes \Z_4$, then $c(H)=9$ and $|G|=20p$. In this case, either $p=3$ or $G$ has a unique subgroup of order $p$. Also, $H$ is a unique subgroup of $G$ of order $20$. Therefore, $G$ has at least $c(H)+3+1=9+3+1$ cyclic subgroups, where $3$ denotes cyclic subgroups of order $p, 2p$ and $5p$ and $1$ denotes the cyclic subgroup of order $10p$. By counting all the elements of $G$, we can see that this group can not be $13$-cyclic.\\
 If $H\cong \Z_3\rtimes \Z_4$, then we can check that no such group is $13$-cyclic by using a similar argument.
\item \textbf{All the Sylow subgroups of $\mathbf{G}$ are not cyclic.} The Sylow $p$-subgroup of $G$ is isomorphic to $\Z_p\times \Z_p$. Since $c(\Z_p\times \Z_p)=p+2$, then $p\leq 7$. First, we assume that Sylow $p$-subgroup of $G$ is not unique. Then, for $p=3,5$ and $7$, we can easily check that, $c(G)>13$. If $p=2$, then $G$ has at least three Sylow $2$-subgroups isomorphic to $\Z_2\times \Z_2$, and $G$ has at least $7$ subgroups of order $2$. Corresponding to this, first assume that both Sylow $q$ and $r$-subgroups of $G$ are not normal. Then, by Sylow theorem, we have $7+1+(1+q)+(1+r)=c(G)\leq 13$, where $7$ denotes the number of cyclic subgroups of order $2$, $1$ denotes the trivial subgroup. This implies that $q+r\leq 3$, which is a contradiction. Further, assume that Sylow $q$-subgroup is normal and Sylow $r$-subgroup is not normal. Then, by Sylow theorem, we get $7+2+(1+r)=c(G)\leq 13$, where $2$ denotes the trivial subgroup and subgroup of order $q$. This shows that $r=3$. Again, by Sylow theorem, $G$ has at least $q$ subgroups of order $2$, which implies that $q=5$ or $7$. By using GAP, one can check that no such group is $13$-cyclic. Finally, assume that the Sylow $q$ and $r$-subgroups of $G$ are normal. Then $G$ also has a unique cyclic subgroup of order $qr$. Now, if $G$ has a cyclic subgroup of order  $2q$ and $2r$, then $7+3+3\leq c(G)=13$, where $7$ denotes the number of cyclic subgroups of order $2$, $3$ denotes the number of cyclic subgroups of order $1,q$ and $r$, and last $3$ denotes the number of cyclic subgroups of order $2q.2r$ and $qr$. By equation~\ref{eq1}, we can see such a group does not exist. Further, assume that a subgroup of order $2q$ is cyclic and a subgroup of order $2r$ is non-cyclic in $G$. Then $G$ contains at least $r$ subgroups of order $2$, that is $r\leq 7$. Now, either $G$ contains a unique cyclic subgroup of order $2q$, which implies that $G$ contains a cyclic subgroup of order $2qr$, or $G$ has at least $2$ cyclic subgroups of order $2q$. In both cases, by counting the number of elements of $G$, and by using the fact $c(G)=13$, we can see that this case is not possible. Similarly, we can show that if $G$ has no cyclic subgroups of order $2q$ and $2r$, then $q,r\leq 7$. By GAP, we can check that no such group of these orders for the obtained choices of $q$ and $r$ is $13$-cyclic.\\
Now assume that $G$ has a unique Sylow $p$-subgroup, and the Sylow $p$-subgroup $P$ is isomorphic to $\Z_7\times \Z_7$. Then $c(P)=9$ and one of the Sylow $q$, $r$ subgroups is normal in $G$, and one is not normal. Without loss of generality, assume that Sylow $q$-subgroup is not normal in $G$. Then $G$ has at least $9+(1+q)+1$ many cyclic subgroups, where $9$ denotes the number of cyclic subgroups of orders $1$ and $p$, $1+q$ denotes the number of subgroups of order $q$, and $1$ denotes a cyclic subgroup of order $r$. By using the fact $c(G)=13$, we get $q=2$. Now, by counting the elements of $G$, we can see that no such group is $13$-cyclic. Similarly, it can be shown that for $p=2,3$ or $5$ leads to a contradiction. We omit the details of calculations for brevity. Hence, the result holds.
    \end{enumerate}
 \end{proof}
 After combining the Lemmas $3.2,3.3,3.4,3.5,3.6$ and $3.7$, we can fully classify all the groups having $13$ cyclic subgroups. The result is as follows.
 \begin{theorem}
     A finite group $G$ is $13$-cyclic if and only if $G\cong \Z_{11}\times \Z_{11},\Z_{p^{12}},D_{22},\\
    SL(2,3),Q_{32},\Z_{11}\rtimes \Z_5, \Z_7\rtimes \Z_8, \Z_5\rtimes \Z_{16}, \Z_3\rtimes \Z_{32}$ or $\Z_7\rtimes \Z_{27}$, where $p$ is a prime number.
 \end{theorem}
 
 \section*{Acknowledgements} 
 The author would like to acknowledge the support of the IISER Berhampur institute post-doctoral fellowship during this work.

\subsection*{Data Availability Statements}
Data sharing is not applicable to this article as no datasets were generated or analyzed during the current study.

\bibliographystyle{abbrv}
\bibliography{refs.bib} 
\end{document}